\documentclass[3p]{elsarticle}
\usepackage{amsmath,amsfonts,amssymb}
\usepackage{nicematrix}
\usepackage{lineno,hyperref}
\biboptions{numbers,sort&compress}
\usepackage{tikz,mathrsfs}
\usepackage{tikz}
\usetikzlibrary{fit}
\usepackage{textcomp}
\modulolinenumbers[5]
\usepackage[titletoc]{appendix}

\journal{System \& Control Letters}
\DeclareMathOperator{\rank}{rank}

\DeclareMathOperator{\im}{im}

\newtheorem{theorem}{Theorem}

\newtheorem{example}{Example}
\newtheorem{lemma}{Lemma}
\newdefinition{remarks}{Remark}
\newdefinition{proposition}{Proposition}
\newdefinition{definition}{Definition}
\newproof{proof}{Proof}

\let\oldmarginpar\marginpar
\renewcommand\marginpar[1]{\-\oldmarginpar[\raggedleft\footnotesize #1]%
{\raggedright\footnotesize #1}}

\newcommand{\R}{{\mathbb{R}}}
\newcommand{\N}{{\mathbb{N}}}

\newcommand{\sB}{\mathscr{B}}

\newcommand{\setdef}[2]{\left\{\, #1 \left|\, \vphantom{#1} #2\right.\right\}}

\begin{document}
	
	\begin{frontmatter}
		\title{Partial impulse observability of linear descriptor systems }
		
		% Group authors per affiliation:
		\author[a]{Juhi Jaiswal}
		\ead{juhi_1821ma03@iitp.ac.in}
		
		\author[b]{Thomas Berger}\ead{thomas.berger@math.upb.de}

		\author[a]{Nutan Kumar Tomar\corref{cor1}}
		\ead{nktomar@iitp.ac.in}
		
		\address[a]{Department of Mathematics, Indian Institute of Technology Patna, India}
		
		\address[b]{Universit\"at Paderborn, Institut f\"ur Mathematik, Warburger Str.~100, 33098~Paderborn, Germany}
		
		\cortext[cor1]{Corresponding author}
		
\begin{abstract}	
A research paper in this journal vol.~61, no.~3, pp.~427--434, 2012, by M. Darouach, provides a functional observer design for linear descriptor systems under the partial impulse observability condition. The observer design is correct, but there was a flaw in the algebraic criterion characterizing partial impulse observability. In the present paper, we derive a novel characterization of partial impulse observability in terms of a simple rank condition involving the system coefficient matrices and an alternative characterization in terms of the Wong sequences.
\end{abstract}
		
\begin{keyword}
Linear descriptor systems, Differential algebraic equations, Partial impulse observability, Wong sequences
\end{keyword}
		
	\end{frontmatter}
	
\section{Introduction}\label{intro}\noindent 	We consider linear time-invariant multivariable descriptor systems of the form
\begin{subequations}\label{dls1}
	\begin{eqnarray}
		E\dot{x}(t) &=& Ax(t), \label{dls1a} \\
		y(t) &=& Cx(t),  \label{dls1b} \\
		z(t) &=& Lx(t), \label{dls1c}
	\end{eqnarray}
\end{subequations}
where $E,A \in \mathbb{R}^{m \times n},~C \in \mathbb{R}^{p \times n}$, and $L \in \mathbb{R}^{r \times n}$ are known constant matrices. We call $x(t) \in \mathbb{R}^n$ the (unknown) semistate vector, $y(t) \in \mathbb{R}^p$ the measured output vector, and $z(t) \in \mathbb{R}^r$ the unknown output vector. The unknown output $z$ contains those variables which cannot be measured and, therefore, observers are required to estimate them.
	
In \cite{darouach2012functional}, Darouach derived an algebraic test for partial impulse observability of \eqref{dls1} with respect to $L$ and used this concept in designing a functional observer to estimate $z$. Under the same algebraic assumption, the observer designing approach of \cite{darouach2012functional} has been improved by using a linear matrix inequality (LMI) formulation in \cite{darouach2017functional}. The concept of partial impulse observability was first introduced in \cite{darouach2012functional} as follows.
	
\begin{definition}\cite[Def.~1]{darouach2012functional}\label{def:Kimpulse}
The descriptor system \eqref{dls1}, or the triplet $(E,A,C)$, is said to be partially impulse observable with respect to $L$, if $y(t)$ is impulse free for $t \geq 0$ only if $Lx(t)$ is impulse free for $t \geq 0$.
\end{definition}
	
In the papers \cite{darouach2012functional} and \cite{darouach2017functional}, the following characterization of partial impulse observability of \eqref{dls1} was provided.
	
\begin{theorem}(\cite[$(1)$~\&~$(5)$ of Lem.~$5$]{darouach2012functional} and \cite[$(1)$~\&~$(2)$ of Lem.~$1$]{darouach2017functional})\label{Darouachtheorem}
The triplet $(E,A,C)$ is partially impulse observable with respect to $L$ if, and only if,
\begin{equation}\label{Kimpulse}
\rank \begin{bmatrix} E & A \\ 0 & E \\ 0 & C \\ 0 & L \end{bmatrix} =
\rank \begin{bmatrix} E & A \\ 0 & E \\ 0 & C \end{bmatrix}.
\end{equation}
\end{theorem}
	
Roughly speaking, partial observability of \eqref{dls1} is related to the reconstruction of $z(t)$ from the knowledge of $y(t)$. Since, corresponding to an inconsistent initial value, $z$ may exhibit impulses, the concept of partial impulse observability is an important aspect for any descriptor system. Thus it is evident that to study  partial impulse observability, we have to consider a proper framework of distributional solutions of~\eqref{dls1}. Here, we consider the class of piecewise-smooth distributions $\mathscr{D}'_{pw\mathscr{C}^{\infty}}$ as introduced in \cite{trenn2009distributional}; for a thorough discussion of this class of distributions, we also refer to \cite{trenn2013solution}. Motivated by~\cite{berger2017observability}, we denote the set of all distributional solutions of~\eqref{dls1} on $[0,\infty)$ by
\[
    \sB := \setdef{(x,y,z) \in (\mathscr{D}'_{pw\mathscr{C}^{\infty}})^{n+p+r} }{ (x,y,z) \text{ satisfies \eqref{dls1} on } [0,\infty) }.
\]
$\sB$ is called ITP-behavior in \cite{berger2017observability}. We stress that it is important that the system \eqref{dls1} is only supposed to hold on $[0,\infty)$, and the solution is free on $(-\infty,0)$, which is different from considering solutions of \eqref{dls1} on $\R$ restricted to  $[0,\infty)$, because of possible impulsive terms due to inconsistent initial values. Furthermore, it is important to note that the distributional restriction to any interval $M \subseteq \mathbb{R}$ is well defined for $\mathcal{D} \in \mathscr{D}'_{pw\mathscr{C}^{\infty}}$, see \cite{trenn2013solution}. Moreover, any $\mathcal{D} \in \mathscr{D}'_{pw\mathscr{C}^{\infty}}$ can be uniquely represented as a combination of a distribution induced by a locally integrable piecewise-smooth function $f$, Dirac delta distributions $\delta_{t_j}$ and their distributional derivatives $\delta_{t_j}^{(i)}$, see~\cite{trenn2009distributional}. The part of $\mathcal{D} \in \mathscr{D}'_{pw\mathscr{C}^{\infty}}$ corresponding to $\delta_{t_j}$ and its derivatives is called the \textit{impulsive part} and denoted by $D[t_j]$, see also the definition in \cite[Eq.~(2)]{berger2017observability}. Since the class $\mathscr{D}'_{pw\mathscr{C}^{\infty}}$ also allows to perform point evaluation of any element, throughout the article, $x[t],~y[t]$, and $z[t]$ stand for the impulsive part of the respective variables at time $t$.

We exploit the behavior $\sB$ to reformulate the definition of partial impulse observability as follows.

\begin{definition}
The descriptor system \eqref{dls1} or the triplet $(E,A,C)$ is partially impulse observable with respect to $L$, if
\[
    \forall\, (x,y,z)\in \mathscr{B}:\ \big(\forall\, t\ge 0:\ y[t]=0\big)\ \implies\ \big(\forall\, t\ge 0:\ z[t]=0\big).
\]
\end{definition}
	
The paper is organized as follows. Section \ref{prelim} collects some preliminary results used in the remainder of the article. In Section \ref{count}, we show that the test condition \eqref{Kimpulse} can give erroneous results, and hence Theorem \ref{Darouachtheorem} is wrong in general. Section \ref{mainresult} contains the main contribution of the paper, where we provide a modified algebraic test to check the partial impulse observability of \eqref{dls1}. Section \ref{numerical} contains a few examples to illustrate the proposed theory. Finally, Section \ref{conc} concludes the paper.
	
We use the following notation: $0$ and $I$ stand for zero and identity matrices of appropriate dimension, respectively. Sometimes, for more clarity, the identity matrix of size $n \times n$ is denoted by $I_n$. In a block partitioned matrix, all missing blocks are zero matrices of appropriate dimensions. The symbols $\im A$ and $\ker A$ denote the image and kernel, respectively, of any matrix $A \in \mathbb{R}^{m \times n}$. The set $AM := \setdef{Ax}{x \in M}$ is the image of a subspace $M \subseteq \R^n$ under $A \in \mathbb{R}^{m \times n}$ and $A^{-1}M := \setdef{x \in \mathbb{R}^n}{Ax \in M}$ represents the pre-image of $M \subseteq \R^m$ under $A \in \mathbb{R}^{m \times n}$.

\section{Preliminaries}\label{prelim}

First we  collect some standard results for the characterization of solutions to the following homogeneous system:
\begin{equation}\label{pre:sys}
\mathscr{E}\dot{x} = \mathscr{A}x,
\end{equation}
where $\mathscr{E},\mathscr{A} \in \mathbb{R}^{m \times n}$. By a solution we mean a piecewise-smooth distribution $x\in(\mathscr{D}'_{pw\mathscr{C}^{\infty}})^{n}$ which satisfies \eqref{pre:sys} on $[0,\infty)$. The first order matrix polynomial $(\lambda \mathscr{E} - \mathscr{A})$, in the indeterminate $\lambda$, is called \textit{matrix pencil} for~\eqref{pre:sys}. Any matrix pencil $(\lambda \mathscr{E}-\mathscr{A})$ is called \textit{regular}, if $m=n$ and $\det(\lambda \mathscr{E}-\mathscr{A}) \neq 0$. If a matrix pencil is not regular, it is called \textit{singular}. For any singular matrix pencil, the Kronecker canonical form (KCF) is the simplest decomposition which provides many useful theoretical tools for analyzing \eqref{pre:sys}.	

\begin{lemma}\cite{gantmacher1959theory}
	{\em The Kronecker Canonical Form (KCF):} For every matrix pencil $(\lambda \mathscr{E} - \mathscr{A})$ there exist nonsingular matrices $P \in \mathbb{C}^{m \times m}$ and $Q \in \mathbb{C}^{n \times n}$ such that, for multi-indices $\epsilon,~f,~\sigma$, and $\eta$,
	\begin{eqnarray}\label{kron}
	P(\lambda \mathscr{E} - \mathscr{A})Q =  \begin{bmatrix}
	\lambda E_{\epsilon} - A_{\epsilon} & & & \\ & \lambda I_f - J_f & & \\ & & \lambda J_{\sigma} - I_{\sigma} & \\ & & & \lambda E_{\eta} - A_{\eta}
	\end{bmatrix},
	\end{eqnarray}
	where $\lambda E_{\epsilon} - A_{\epsilon}$ and $\lambda E_{\eta} - A_{\eta}$ have block diagonal structure; each block takes the form $\lambda E_{{\epsilon}_i} - A_{{\epsilon}_i} = \lambda \begin{bmatrix} I_{{\epsilon}_i} & 0_{{\epsilon}_i \times 1} \end{bmatrix} - \begin{bmatrix} 0_{{\epsilon}_i \times 1} & I_{{\epsilon}_i} \end{bmatrix} $ and $\lambda E_{{\eta}_i} - A_{{\eta}_i} = \lambda \begin{bmatrix} I_{{\eta}_i} \\ 0_{1 \times {\eta}_i} \end{bmatrix} - \begin{bmatrix} 0_{1 \times {\eta}_i} \\ I_{{\eta}_i} \end{bmatrix} $, respectively; both $J_f$ and $J_{\sigma}$ are in Jordan canonical form; $J_{\sigma}$ has zeros on its diagonal and thus is a nilpotent matrix; $J_f$ contains, on its diagonal, all finite eigenvalues of $(\lambda \mathscr{E} - \mathscr{A})$.
\end{lemma}

\begin{remarks}\label{rem1}
The blocks in \eqref{kron} appear only in pairs. For example, if $E_{\epsilon}$ vanishes, then $A_{\epsilon}$ also vanishes. Moreover, $\epsilon-$blocks with $\epsilon_i = 0$ and/or $\eta-$blocks with $\eta_i = 0$ are possible, which results in zero columns (for $\epsilon_i = 0$) and/or zero rows (for $\eta_i = 0$) in the KCF \eqref{kron}. The KCF structure \eqref{kron} is unique up to the reordering of the diagonal blocks. The KCF \eqref{kron} without $\epsilon-$ and $\eta-$blocks is also called the Weierstrass canonical form (WCF). In case of a regular matrix pencil $(\lambda \mathscr{E}-\mathscr{A})$, the KCF \eqref{kron} reduces to the WCF.
\end{remarks}

\begin{remarks}\label{rem3}
	In this paper, we use the KCF \eqref{kron} to simplify the proof of some theoretical results. But, the determination of the KCF is not recommended because the computation is numerically ill-posed~\cite{van1979computation}. Furthermore, since $J_f$ and $J_{\sigma}$ are in Jordan canonical form, in general, the matrices $P$ and $Q$ in \eqref{kron} are complex-valued matrices. This is computationally undesirable, because if the system matrices are real-valued, one would like to get real $P$ and $Q$. To remove such difficulties in the computation of the KCF, based on the Wong sequences, a numerically stable quasi-Kronecker decomposition, which also reveals the KCF structure, can be found in \cite{berger2012quasi,berger2013addition}.
\end{remarks}

The solution theory of descriptor systems is a simple application of the KCF because it has a block diagonal structure and the associated variables can be considered separately. Setting
\begin{equation}\label{S2}
x = Q\begin{bmatrix}
x_{\epsilon}^\top & x_f^\top & x_{\sigma}^\top & x_{\eta}^\top
\end{bmatrix}^\top,
\end{equation}
then in terms of the four different blocks in the KCF, \eqref{pre:sys} is transformed into
\begin{subequations}\label{IV.1}
	\begin{eqnarray}
	E_{\epsilon}\dot{x}_{\epsilon} &=& A_{\epsilon} x_{\epsilon} \label{IV.1a}, \\
	\dot{x}_f &=& J_fx_f \label{IV.1b},\\
	J_{\sigma}\dot{x}_{\sigma} &=& x_{\sigma} \label{IV.1c}, \\
	E_{\eta} \dot{x}_{\eta} &=& A_{\eta}x_{\eta} \label{IV.1d}.
	\end{eqnarray}
\end{subequations}
The following solution analysis of \eqref{pre:sys}, via \eqref{kron}, is now straightforward.
\begin{enumerate}
	\item[S$1$)] Systems of the form \eqref{IV.1a}
	can be written as
	\begin{equation}\label{S4}
	\begin{bmatrix} I_{\epsilon} & 0 \end{bmatrix} \begin{bmatrix} \dot x_1\\ \dot x_2\end{bmatrix} = \begin{bmatrix} A_1 & A_2 \end{bmatrix} \begin{bmatrix} x_1\\ x_2\end{bmatrix},
	\end{equation}
where $A_1$ is a nilpotent matrix. Thus any solution $x_{\epsilon} = \begin{bmatrix} x_1^\top & x_2^\top \end{bmatrix}^\top$ to \eqref{S4} is given by
\begin{equation}\label{S1}
  \begin{bmatrix} x_1(t) \\ x_2(t) \end{bmatrix} = \begin{bmatrix} e^{A_1 t} x_1^0 + \int_0^t e^{A_1(t-\tau)} A_2 x_2(\tau) {\rm d}\tau \\ x_2(t) \end{bmatrix} = \begin{bmatrix}
  e^{A_1t} x_1^0 + \sum_{i = 0}^{h_1-1} A_1^i A_2 \int_0^t \frac{(t-\tau)^i}{i!} x_2(\tau) {\rm d}\tau \\ x_2(t) \end{bmatrix},\ t\ge 0,
\end{equation}	
for some $x_1^0$ of appropriate dimension,	where $h_1$ is the nilpotency index of $A_1$ and $x_2$ is arbitrary. Hence, in general, $x_{\epsilon}$ satisfying \eqref{IV.1a} is always impulsive, cf. \cite[p.~$26$]{berger2017observability}. Moreover, by \cite[Cor.~$2.4$]{trenn2013solution} any solution $x$ of \eqref{pre:sys} is uniquely determined if, and only if, the $\epsilon-$blocks in \eqref{kron} are not present.
\item[S$2$)] Corresponding to any initial condition, the solution of the free homogeneous state space system \eqref{IV.1b} exhibits no impulses, see \cite[Thm.~$3.3$]{trenn2009distributional}, i.e., $x_f[t] = 0$ for all $t\ge 0$.
\item[S$3$)] According to \cite{dai1989singular} the solution of \eqref{IV.1c} is given by
\begin{equation}\label{S3}
	x_{\sigma}|_{[0,\infty)} = -\sum_{i = 1}^{h_2-1} \delta^{(i-1)}J_{\sigma}^i x_{\sigma}(0-),
\end{equation}
where $h_2$ is the nilpotency index of the matrix $J_{\sigma}$. Therefore, the solution of \eqref{IV.1c} is impulsive (i.e., $x_{\sigma}[0]\neq 0$) if, and only if, $x_{\sigma}(0-) \notin \ker J_{\sigma}$.
\item[S$4$)] Each block in \eqref{IV.1d} can be written as
\begin{eqnarray*}
\dot{x}_{\eta_i} &=& J_{\eta_i}^\top x_{\eta_i}, \\ 0 &=& e_{\eta_i}^\top x_{\eta_i},
\end{eqnarray*}
where $J_{\eta_i}^\top$ is a nilpotent matrix having nilpotency index $\eta_i$ and $e_{\eta_i}$ is the last column of $I_{\eta_i}$. The only solution for this block is $x_{\eta} = 0$ and, in particular, $x_{\eta}[t] = 0$ for all $t\ge 0$, cf.\ also~\cite[p.~$25$]{berger2017observability}. Consequently, there are no impulses in the solutions of \eqref{pre:sys} due to $\eta-$blocks.
\end{enumerate}

\begin{remarks}
From above solution analysis, it is clear that the semistate $x$ in~\eqref{dls1} may have impulses only due to $\epsilon$- and $\sigma$-blocks in the KCF of $(\lambda E - A)$.
\end{remarks}

The concept of partial impulse observability of \eqref{dls1} is a natural extension of impulse observability (I-observability) of a system \eqref{dls1a}-\eqref{dls1b}: $(E,A,C)$ is impulse observable if, and only if, $(E,A,C)$ is partially impulse observable with respect to $L=I_n$. Alternative definitions for impulse observability are given in~\cite{hou1999causal,ishihara2001impulse} for instance, see also the survey \cite{berger2017observability} for more details. To check the I-observability of system \eqref{dls1a}-\eqref{dls1b}, the following algebraic criterion has been provided in the literature~\cite{berger2017observability,hou1999causal,ishihara2001impulse}:
\begin{equation}\label{Iobsv}
\rank \begin{bmatrix} E & A \\ 0 & E \\ 0 & C \end{bmatrix} = n + \rank E.
\end{equation}

\begin{remarks}\label{rem5}
Clearly, I-observability of \eqref{dls1a}-\eqref{dls1b} implies partial impulse observability of \eqref{dls1} with respect to any matrix $L$. But the opposite implication is not true in general. In Remark \ref{Iobsvw} below, we show that when $L=I_n$, the criteria  for partial impulse observability of \eqref{dls1} developed in the following sections trivially reduce to I-observability of \eqref{dls1a}-\eqref{dls1b}.
\end{remarks}

Now, we present some results from basic linear algebra, which play an important role in the further discussion.
The following fundamental result can be found in any standard textbook on linear algebra.

\begin{lemma}\label{lm:ker}
	Let $X$ and $Y$ be any two matrices of compatible dimensions. Then $\rank \begin{bmatrix} X \\ Y \end{bmatrix} = \rank X$ if, and only if, $\ker X \subseteq \ker Y$.
\end{lemma}

\begin{lemma}\cite{matsaglia1974equalities}\label{lm:rank}
	Let $X ,~W$, and $Y$ be any matrices of compatible dimensions. If $X$ has full row rank and/or $Y$ has full column rank, then
	$$\rank \begin{bmatrix}
	X & W \\ 0 & Y
	\end{bmatrix} = \rank{X} + \rank{Y}.$$
\end{lemma}

Finally, we recall the concept of Wong sequences corresponding to \eqref{dls1a} from~\cite{berger2012quasi}; for our purposes we only need the second Wong sequence.

\begin{definition}\label{def:wong}
For matrices $E,A\in\R^{m\times n}$ the Wong sequence $\{\mathcal{W}_{[E,A]}^i\}_{i = 0}^{\infty}$ is a sequence of subspaces, defined by
\[
	\mathcal{W}_{[E,A]}^0 := \{0\}, \quad	\mathcal{W}_{[E,A]}^{i+1} := A^{-1}(E\mathcal{W}_{[E,A]}^i),\ i\in\N.
\]
The union $\mathcal{W}_{[E,A]}^* := \bigcup_{i \in \mathbb{N}} \mathcal{W}_{[E,A]}^i$ is called the limit of the Wong sequence.
\end{definition}

\noindent We conclude this section by recalling the following result for I-observability of system \eqref{dls1a}-\eqref{dls1b} in terms of the Wong sequences.

\begin{lemma}\cite{berger2017observability}\label{lm4}
	The triple $(E,A,C)$ is I-observable if, and only if,
	\begin{equation*}
	\mathcal{W}^*_{[\bar{E},\bar{A}]} \cap \bar{A}^{-1}(\im{\bar{E}}) = \{0\},
	\end{equation*}
where $\bar{E} = \begin{bmatrix} E \\ 0 \end{bmatrix}$ and $\bar{A} = \begin{bmatrix} A \\ C \end{bmatrix}$.
\end{lemma}

\section{Counterexample}\label{count}

In this section, an example is given to show that the algebraic condition \eqref{Kimpulse} is not equivalent to the partial impulse observability of \eqref{dls1}, proving Theorem \ref{Darouachtheorem} wrong.

Consider system \eqref{dls1} with matrices:
$$E = \begin{bmatrix} 0 & 1 & 0  \\ 0 & 0 & 1  \\ 0 & 0 & 0  \end{bmatrix},~
A = \begin{bmatrix} 1 & 0 & 0  \\ 0 & 1 & 0  \\ 0 & 0 & 1  \end{bmatrix},~
C = \begin{bmatrix} 0 & 0 & 1 \end{bmatrix}, \text{ and } L = \begin{bmatrix} 0 & 1 & 0 \end{bmatrix}.$$
Here it is straightforward to see that the condition \eqref{Kimpulse} is satisfied.  Since this system is regular and of the form~\eqref{IV.1c} we obtain from~\eqref{S3} that
\begin{eqnarray}\label{soln}
x|_{[0,\infty)} = - \delta Ex(0-) - \dot \delta E^2x(0-)
\end{eqnarray}
for any $(x,y,z)\in\sB$. Since $CE =CE^2 =  0 $, \eqref{soln} gives that $y[t] = C x[t] = 0$ for all $t\ge 0$, but $z[0] = -\delta \begin{bmatrix} 0 & 0 & 1 \end{bmatrix} x(0-)$, which may be nonzero for some $x(0-)$. Therefore, the system is not partially impulse observable with respect to $L$. This proves that, in general, Theorem~\ref{Darouachtheorem} is not correct.

\section{Main result}\label{mainresult}

The main aim of this section is to derive a simple rank criteria for partial impulse observability of \eqref{dls1} in terms of the original system matrices. In Theorem \ref{thm1} below, we first derive this condition on the basis of the KCF of the matrix pencil $(\lambda E - A)$. To prove Theorem \ref{thm1}, without loss of generality, we assume that the pencil $(\lambda E - A)$ in system \eqref{dls1} is in KCF. Moreover, we use the notation
\begin{equation}\label{CQ}
CQ = \begin{bmatrix} C_{\epsilon} & C_f & C_{\sigma} & C_{\eta} \end{bmatrix}
\end{equation}
and
\begin{equation}\label{LQ}
LQ = \begin{bmatrix} L_{\epsilon} & L_f & L_{\sigma} & L_{\eta} \end{bmatrix},
 \end{equation}
where the sizes of the block matrices on the right hand side of \eqref{CQ} and \eqref{LQ} are compatible with the sizes of the blocks in the KCF of $(\lambda E - A)$. Furthermore,
we assume that $C_{\epsilon}$ and $L_{\epsilon}$ are partitioned as follows, corresponding to the decomposition $x_{\epsilon} = \begin{bmatrix}
x_1^\top & x_2^\top \end{bmatrix}^\top$ as in \eqref{S4}:
\begin{equation}\label{c1}
C_{\epsilon} = \begin{bmatrix} C_1 & C_2 \end{bmatrix}\quad \text{ and }\quad
L_{\epsilon} = \begin{bmatrix} L_1 & L_2 \end{bmatrix}.
\end{equation}

\begin{theorem}\label{thm1}
Consider system \eqref{dls1}.  Then $(E,A,C)$ is partially impulse observable with respect to $L$ if, and only if,
\begin{equation}\label{eq:cond}
 \ker \begin{bmatrix}
	C_2 & C_1A_2 & C_1A_1A_2 & \hdots  &
	C_1A_1^{l-1}A_2 & -C_{\sigma}J_{\sigma} \\
	& C_2 & C_1A_2 &  \hdots &
	C_1A_1^{l-2}A_2 & -C_{\sigma}J_{\sigma}^2 \\
	& & \ddots & \ddots  & \vdots & \vdots \\
	 & & & C_2 & C_1A_2 & -C_{\sigma}J_{\sigma}^{l}\\
	 & & & & C_2 & 0
   \end{bmatrix} \subseteq \ker \begin{bmatrix}
	L_2 & L_1A_2 & L_1A_1A_2 & \hdots &
	L_1A_1^{l-1}A_2 & -L_{\sigma}J_{\sigma} \\
	& L_2 & C_1A_2 &  \hdots &
	L_1A_1^{l-2}A_2 & -L_{\sigma}J_{\sigma}^2 \\
	& & \ddots & \ddots & \vdots & \vdots \\
	 & & & L_2 & L_1A_2 & -L_{\sigma}J_{\sigma}^{l}\\
	 & & & & L_2 & 0 \end{bmatrix}
\end{equation}
for all $l \geq n+1$.
\end{theorem}

\begin{proof}
	$(\Rightarrow)$: Let $l \geq n+1$ and
	\begin{equation*}
	\begin{bmatrix}
	v_{0} \\ v_{1}  \\ v_{2} \\ \vdots  \\ v_{l} \\ v \end{bmatrix}  \in \ker \begin{bmatrix}
	C_2 & C_1A_2 & C_1A_1A_2 & \hdots &
	C_1A_1^{l-1}A_2 & -C_{\sigma}J_{\sigma} \\
	& C_2 & C_1A_2 &  \hdots  &
	C_1A_1^{l-2}A_2 & -C_{\sigma}J_{\sigma}^2 \\
	& & \ddots & \ddots &  \vdots & \vdots \\
	& & & C_2 & C_1A_2 & -C_{\sigma}J_{\sigma}^{l} \\ & & & & C_2 & 0
	\end{bmatrix}.
	\end{equation*}
Define $x_{\sigma}(t) = v$ for $t<0$, $x_\sigma|_{[0,\infty)}$ as in \eqref{S3} and
$x_2 = \sum_{j=0}^l \delta^{(j)} v_j$. Then, with $x_1$ as in \eqref{S1} for $x_1^0 = 0$, $x_{\epsilon} = \begin{bmatrix} x_1^\top & x_2^\top \end{bmatrix}^\top$, $x_f = 0$, $x_{\eta} = 0$, $y = C_1x_1+C_2x_2+C_{\sigma}x_{\sigma}$, and $z = L_1x_1+L_2x_2+L_{\sigma}x_{\sigma}$ we have that $(x,y,z)\in \mathscr{B}$.
Then by using the convolution property
\begin{eqnarray}\label{D1}
 \int_0^t \frac{(t-\tau)^i}{i!} \delta^{(j)}_{s} {\rm d}\tau &=& \begin{cases} \frac{(t-s)^{i-j}}{(i-j)!}, & j=0,\ldots,i,\\[3mm]
\delta^{(j-i-1)}_{s}, & j = i+1,\ldots,l,\end{cases}
\end{eqnarray}
for any $s\ge 0$, the equation \eqref{S1} implies
\begin{equation}\label{xsol}
x_1(t) = \sum_{i=0}^{h_1-1} A_1^i A_2 \left\{ \sum_{j=0}^{i} \frac{t^{i-j}}{(i-j)!} v_{j} + \sum_{j=i+1}^{l} \delta^{(j-i-1)} v_{j} \right\},
\end{equation}
where $h_1$ and $h_2$ are the nilpotency indices of $A_1$ and $J_{\sigma}$, respectively.
Thus, by \eqref{xsol} and \eqref{S3}, we obtain
\begin{subequations}\label{x}
	\begin{eqnarray}
	   x_1[0] &=& \sum_{i=0}^{h_1-1} \sum_{j=i+1}^{l} \delta^{(j-i-1)} A_1^i A_2 v_{j}, \label{x10} \\
		x_{\sigma}[0] &=& -\sum_{i=0}^{h_2-1} J_\sigma^{i+1} \delta^{(i)}v \label{xsigma}
	\end{eqnarray}
\end{subequations}
and clearly $x_1[t]=0$ and $x_{\sigma}[t]=0$ for all $t>0$.
Since, by choices of $v$ and $v_i~(0 \leq i \leq l)$,
	\begin{eqnarray*}
	y[0] &=& \sum_{i=0}^{h_1-1} \sum_{j=i+1}^{l} \delta^{(j-i-1)} C_1A_1^i A_2 v_{j} + \sum_{i=0}^{l} \delta^{(i)} C_2v_{i} - \sum_{i=0}^{l} \delta^{(i)} C_{\sigma}J_\sigma^{i+1}v \\
   &=& \begin{bmatrix}
	\delta I & \delta^{(1)} I & \hdots &  \delta^{(l)} I
	\end{bmatrix} \begin{bmatrix}
	C_2 & C_1A_2 & C_1A_1A_2 & \hdots &
	C_1A_1^{l-1}A_2 & -C_{\sigma}J_{\sigma} \\
	& C_2 & C_1A_2 &  \hdots &
	C_1A_1^{l-2}A_2 & -C_{\sigma}J_{\sigma}^2 \\
	& & \ddots & \ddots  & \vdots & \vdots \\
	& & & C_2 & C_1A_2 & -C_{\sigma}J_{\sigma}^{l} \\ & & & & C_2 & 0
	\end{bmatrix}\begin{bmatrix}
	v_{0} \\ v_{1}  \\ v_{2} \\ \vdots  \\ v_{l} \\ v
	\end{bmatrix} = 0,
	\end{eqnarray*}
partial impulse observability of the system implies $z[0] = 0$. Thus
\begin{eqnarray}
	0 = z[0] &=& \sum_{i=0}^{h_1-1} \sum_{j=i+1}^{l} \delta^{(j-i-1)} L_1A_1^i A_2 v_{j} + \sum_{i=0}^{l} \delta^{(i)} L_2v_{i} - \sum_{i=0}^{l} \delta^{(i)} L_{\sigma}J_\sigma^{i+1}v \nonumber \\
	&=& \begin{bmatrix}
	\delta I & \delta^{(1)} I & \hdots &  \delta^{(l)} I
	\end{bmatrix} \begin{bmatrix}
	L_2 & L_1A_2 & L_1A_1A_2 & \hdots &
	L_1A_1^{l-1}A_2 & -L_{\sigma}J_{\sigma} \\
	& L_2 & L_1A_2 &  \hdots &
	L_1A_1^{l-2}A_2 & -L_{\sigma}J_{\sigma}^2 \\
	& & \ddots & \ddots & \vdots & \vdots \\
	& & & L_2 & L_1A_2 & -L_{\sigma}J_{\sigma}^{l} \\ & & & & L_2 & 0
	\end{bmatrix}\begin{bmatrix}
	v_{0} \\ v_{1}  \\ v_{2} \\ \vdots  \\ v_{l} \\ v \end{bmatrix} . \label{z}
	\end{eqnarray}
This means that
	\begin{equation*}
	\begin{bmatrix}
	v_{0} \\ v_{1}  \\ v_{2} \\ \vdots  \\ v_{l} \\ v \end{bmatrix}  \in \ker \begin{bmatrix}
	L_2 & L_1A_2 & L_1A_1A_2 & \hdots &
	L_1A_1^{l-1}A_2 & -L_{\sigma}J_{\sigma} \\
	& L_2 & L_1A_2 & \hdots &
	L_1A_1^{l-2}A_2 & -L_{\sigma}J_{\sigma}^2 \\
	& & \ddots & \ddots & \vdots & \vdots \\
	& & & L_2 & L_1A_2 & -L_{\sigma}J_{\sigma}^{l} \\ & & & & L_2 & 0
	\end{bmatrix}.
	\end{equation*}
	
\noindent $(\Leftarrow)$: Let $(x,y,z)\in\mathscr{B}$ be such that $y[t] = 0$ for all $t \geq 0$ and partition
$x =  \begin{bmatrix}
x_1^\top & x_2^\top & x_f^\top & x_{\sigma}^\top & x_{\eta}^\top
\end{bmatrix}^\top$ as in \eqref{S2} and \eqref{S4}.
By definition of $\mathscr{D}'_{pw\mathscr{C}^{\infty}}$ there is a locally finite set $(t_k)_{k\in\mathbb{Z}}\subseteq \mathbb{R}$ such that $x_2[t_k] \neq 0$ and $x_2[t] = 0$ for all $t\neq t_k$, see~\cite{trenn2009distributional}. Let $n_1$ and $n_2$ be the number of components in $x_2$ and $x_{\sigma}$, respectively. Then by \cite[Prop.~2.1.12]{trenn2009distributional} there exist $l_k\in \mathbb{N}$ and $v_{k,j}\in\mathbb{R}^{n_1}$ for $k \in \mathbb{Z}$ and $j=0,\ldots,l_k$ such that
 $$x_2[t_k] =  \sum_{j=0}^{l_k} v_{k,j} \delta^{(j)}_{t_k}.$$
Fix $k\in\mathbb{Z}$. Without loss of generality we may assume that $l_k \geq n+1$, otherwise we may add additional terms with $v_{k,j}=0$. Then, by \eqref{S1}, \eqref{S3}, and \eqref{D1} with $s=t_k$, we find that
\begin{eqnarray*}
x_1[t_k] &=& \sum_{i=0}^{h_1-1} \sum_{j=i+1}^{l_k} \delta^{(j-i-1)}_{t_k} A_1^i A_2 v_{k,j}, \\
x_2[t_k] &=& \sum_{i=0}^{l_k} \delta^{(i)}_{t_k} v_{k,i},  \\
x_{\sigma}[0] &=& -\sum_{i=0}^{h_2-1}  \delta^{(i)} J_\sigma^{i+1} x_\sigma^0,
\end{eqnarray*}
where $x_\sigma^0 \in \mathbb{R}^{n_2}$.
Thus from $y[t_k]=0$ it follows that
\begin{eqnarray*}
	\begin{bmatrix}
	v_{k,0} \\ v_{k,1} \\ v_{k,2} \\ \vdots \\ v_{k,l_k} \\ x_\sigma^0 \end{bmatrix}  \in \ker \begin{bmatrix}
	C_2 & C_1A_2 & C_1A_1A_2 & \hdots &
	C_1A_1^{l_k-1}A_2 & -C_{\sigma}J_{\sigma} \\
	& C_2 & C_1A_2 &  \hdots &
	C_1A_1^{l_k-2}A_2 & -C_{\sigma}J_{\sigma}^2 \\
	& & \ddots & \ddots & \vdots & \vdots \\
	& & & C_2 & C_1A_2 & -C_{\sigma}J_{\sigma}^{l_k} \\ & & & & C_2 &0
	 \end{bmatrix}.
	\end{eqnarray*}
Then assumption \eqref{eq:cond} implies that
\begin{eqnarray*}
	\begin{bmatrix}
	v_{k,0} \\ v_{k,1}  \\ v_{k,2} \\ \vdots  \\ v_{k,l_k} \\ x_\sigma^0 \end{bmatrix}  \in \ker \begin{bmatrix}
	L_2 & L_1A_2 & L_1A_1A_2 & \hdots &
	L_1A_1^{l_k-1}A_2 & -L_{\sigma}J_{\sigma} \\
	& L_2 & L_1A_2 &  \hdots &
	L_1A_1^{l_k-2}A_2 & -L_{\sigma}J_{\sigma}^2 \\
	& & \ddots & \ddots & \vdots & \vdots \\
	& & & L_2 & L_1A_2 & -L_{\sigma}J_{\sigma}^{l_k} \\ & & & & L_2 & 0
	\end{bmatrix}
\end{eqnarray*}
which, by a similar calculation as in \eqref{z}, means $z[t_k]= 0$. Since $k$ was arbitrary and $z[t]=0$ for $t\neq t_k$ is obvious, this proves partial impulse observability of \eqref{dls1} with respect to $L$. $\hfill$ $\Box$
\end{proof}

Before investigating the algebraic criteria for partial impulse observability of \eqref{dls1} with respect to $L$, we define
\begin{eqnarray*}
&\bar{E} = \begin{bmatrix} E \\ 0 \end{bmatrix},~\bar{A} = \begin{bmatrix} A \\ C \end{bmatrix},~\bar{E}_1 = \begin{bmatrix} \bar{E} \\ 0 \end{bmatrix},~\bar{A}_1 = \begin{bmatrix} \bar{A} \\ L \end{bmatrix},& \\
& \mathcal{F}_l :=  \setcounter{MaxMatrixCols}{15}
\NiceMatrixOptions
{nullify-dots,code-for-last-col = \color{black},code-for-last-col=\color{black}}
\begin{bNiceMatrix}[first-row,last-col]
& \Ldots[line-style={solid,<->},shorten=0pt]^{l \text{ block columns}} \\
\bar{E} & \bar{A} &  &  &  & \\
& \bar{E} & \bar{A} &  & & \\
&  & \ddots & \ddots &   & \Vdots[line-style={solid,<->}]^{l \text{ block rows}} \\
&  &  & \bar{E} & \bar{A} & \\
&  &  &  & \bar{E} &
\end{bNiceMatrix} ~~,\quad
\mathcal{F}_{l,L} :=  \setcounter{MaxMatrixCols}{15}
\NiceMatrixOptions
{nullify-dots,code-for-last-col = \color{black},code-for-last-col=\color{black}}
\begin{bNiceMatrix}[first-row,last-col]
& \Ldots[line-style={solid,<->},shorten=0pt]^{l \text{ block columns}} \\
\bar{E}_1 & \bar{A}_1 &  &  &  & \\
& \bar{E}_1 & \bar{A}_1 &  & & \\
&  & \ddots & \ddots &   & \Vdots[line-style={solid,<->}]^{l \text{ block rows}}\\
&  &  & \bar{E}_1 & \bar{A}_1 & \\
&  &  &  & \bar{E}_1 &
\end{bNiceMatrix}~~,&
\end{eqnarray*}
and introduce the following rank condition
\begin{equation}\label{eq:kimpulse}
\forall\,l \geq n+1:\ \rank \mathcal{F}_{l} = \rank \mathcal{F}_{l,L}.
\end{equation}

\noindent We now transform the condition \eqref{eq:kimpulse} in terms of the KCF blocks. If $l = 2$, then

\begin{eqnarray*}
	\rank \mathcal{F}_2 = \rank \begin{bmatrix}
		\bar{E} & \bar{A} \\ & \bar{E} \end{bmatrix} = \rank \begin{bmatrix}
		E & A \\ & C \\ & E \end{bmatrix} =
	\rank \begin{bmatrix}
		E_{\epsilon} & & & & A_{\epsilon} \\ & I_f & & & & J_f & & \\ & & J_{\sigma} & & & & I_{\sigma} & \\ & & & E_{\eta} & & & & A_{\eta} \\ & & & & C_{\epsilon} & C_f & C_{\sigma} & C_{\eta} \\ & & & & E_{\epsilon} \\ & & & & & I_f & & \\ & & & & & & J_{\sigma} & \\ & & & & & & & E_{\eta} \end{bmatrix}.
\end{eqnarray*}

\noindent Since $E_{\epsilon}$ has full row rank, $E_{\eta}$ has full column rank, and $I_f$ has full rank, it is a direct consequence of Lemma~\ref{lm:rank} (applied twice) that
\begin{eqnarray*}
\rank \mathcal{F}_2 = \rank E_{\epsilon} +
   2 \rank I_f + 2 \rank E_{\eta} + \rank \begin{bmatrix} J_{\sigma} & & I_{\sigma} \\ & C_{\epsilon} & C_{\sigma} \\ & E_{\epsilon} & \\ & & J_{\sigma} \end{bmatrix} .
\end{eqnarray*}

\noindent Now, by using column operations corresponding to the multiplication of the last matrix in the above identity with $\begin{bmatrix}
-I_{\sigma} \\ & I \\ J_{\sigma} & & I_{\sigma} \end{bmatrix}$ from the right, we obtain
\begin{eqnarray*}
	\rank \mathcal{F}_2	&=& \rank E_{\epsilon} + 2 \rank I_f + 2 \rank E_{\eta} + \rank \begin{bmatrix}  & & I_{\sigma} \\ C_{\sigma}J_{\sigma} & C_{\epsilon} & C_{\sigma} \\ & E_{\epsilon} & \\ J_{\sigma}^2 & & J_{\sigma} \end{bmatrix} \\
	&=& \rank E_{\epsilon} + 2 \rank I_f + 2 \rank E_{\eta} + \rank I_{\sigma} + \rank \begin{bmatrix} C_{\sigma}J_{\sigma} & C_{\epsilon} \\ & E_{\epsilon} \\ J_{\sigma}^2 &  \end{bmatrix}.
\end{eqnarray*}

\noindent Now by substituting $E_{\epsilon} = \begin{bmatrix} I_{\epsilon} & 0 \end{bmatrix},~
C_{\epsilon} = \begin{bmatrix} C_1 & C_2 \end{bmatrix}$ and again using Lemma \ref{lm:rank} due to full rank of $I_{\epsilon}$, we obtain
\begin{eqnarray}
%&= &  \rank E_{\epsilon} + 2 \rank I_f + 2 \rank E_{\eta} + \rank I_{\sigma} + \rank \begin{bmatrix} C_{\sigma}J_{\sigma} & C_1 & C_2 \\ & I_{\epsilon} & 0 \\ J_{\sigma}^2 &  \end{bmatrix} \nonumber \\
\rank \mathcal{F}_2	&= & 2 \rank E_{\epsilon} + 2 \rank I_f + 2 \rank E_{\eta} + \rank I_{\sigma} + \rank \begin{bmatrix} C_{\sigma}J_{\sigma} & C_2 \\ J_{\sigma}^2 & \end{bmatrix}.\label{r1}
\end{eqnarray}
Similarly, it is easy to show that
\begin{eqnarray}
\rank \mathcal{F}_{2,L} &=& \rank \begin{bmatrix}
		\bar{E}_1 & \bar{A}_1 \\ & \bar{E}_1 \end{bmatrix} = \rank \begin{bmatrix}
	E & A \\ & C \\  & L \\ & E \end{bmatrix} \nonumber \\
	&=&  2 \rank E_{\epsilon} + 2 \rank I_f + 2 \rank E_{\eta} + \rank I_{\sigma} + \rank \begin{bmatrix} C_{\sigma}J_{\sigma} & C_2 \\ L_{\sigma}J_{\sigma} & L_2 \\ J_{\sigma}^2 \end{bmatrix}.\label{r2}
\end{eqnarray}

\noindent Thus, in view of Lemma \ref{lm:ker}, \eqref{r1} and \eqref{r2} provide that $\rank \mathcal{F}_2 = \rank \mathcal{F}_{2,L}$ if, and only if,
\begin{equation}\label{R1}
\ker \begin{bmatrix} C_{\sigma}J_{\sigma} & C_2 \\ J_{\sigma}^2 & \end{bmatrix} \subseteq \ker \begin{bmatrix} L_{\sigma}J_{\sigma} & L_2  \end{bmatrix}.
\end{equation}

\noindent Similarly, $\rank \mathcal{F}_3 = \rank \mathcal{F}_{3,L}$ if, and only if,
\begin{equation}\label{R2}
\ker \begin{bmatrix} C_{\sigma}J_{\sigma} & C_2 & C_1A_2 \\ C_{\sigma}J_{\sigma}^2 & & C_2 \\J_{\sigma}^3 & & \end{bmatrix} \subseteq \ker \begin{bmatrix} L_{\sigma}J_{\sigma} & L_2 & L_1A_2 \\ L_{\sigma}J_{\sigma}^2 & & L_2  \end{bmatrix}.
\end{equation}

\noindent Finally, we expound the calculation for $l= 4$ and denote
\begin{equation*}
\bar{U} = \begin{bmatrix}
I_{\sigma} \\ & I \\ -J_{\sigma} & & I_{\sigma} \\ & & & I \\ J_{\sigma}^2 & & -J_{\sigma} & & I_{\sigma} \\ & & & & & I \\ -J_{\sigma}^3 & & J_{\sigma}^2 & & -J_{\sigma} & & I_{\sigma}
\end{bmatrix}\quad \text{ and }\quad \hat{U} = \begin{bmatrix}
I & -C_1 & & C_1A_1 & & -C_1A_1^2 & \\
& I_{\epsilon} & & -A_1 & & A_1^2 & \\
&& I & -C_1 & & C_1A_1 & & \\ &&& I_{\epsilon} & & -A_1 & \\ &&&& I & -C_1 & \\ &&&&& I_{\epsilon} \\ &&&&&&I \end{bmatrix}.
\end{equation*}
We first use the KCF and then apply Lemma \ref{lm:rank} for full rank matrices $E_{\epsilon},~E_{\eta}$, and $I_f$ to produce that
\begin{eqnarray*}
	\rank \mathcal{F}_4 &=& \rank \begin{bmatrix}
\bar{E} & \bar{A} \\ & \bar{E} & \bar{A} \\ && \bar{E} & \bar{A} \\ &&& \bar{E}   \end{bmatrix} = \rank \begin{bmatrix} E & A \\ & C \\ & E & A \\ & & C \\& & E & A \\ &&& C \\ &&& E \end{bmatrix} \\
 &=& \rank E_{\epsilon} + 4 \rank I_f + 4 \rank E_{\eta} + \rank \begin{bmatrix} J_{\sigma} & & I_{\sigma} \\ & C_{\epsilon} & C_{\sigma} \\ & E_{\epsilon} & & A_{\epsilon} \\ & & J_{\sigma} & & I_{\sigma} \\ & & & C_{\epsilon} & C_{\sigma} \\ & & & E_{\epsilon} & & A_{\epsilon} \\ & & & & J_{\sigma} & & I_{\sigma} \\ & & & & & C_{\epsilon} & C_{\sigma} \\ & & & & & E_{\epsilon} & \\ & & & & & & J_{\sigma} \end{bmatrix}.
\end{eqnarray*}
Now, by multiplying the last matrix in the equality above with $\bar{U}$ from the right and again using Lemma~\ref{lm:rank} for $I_{\sigma}$, we obtain
\begin{eqnarray*}
\rank \mathcal{F}_4 &=& \rank E_{\epsilon} + 4 \rank    I_f + 4 \rank E_{\eta} + 3 \rank I_{\sigma} + \rank \begin{bmatrix} C_{\sigma}J_{\sigma} & C_{\epsilon} &  \\ & E_{\epsilon} & A_{\epsilon} \\ C_{\sigma}J_{\sigma}^2 & & C_{\epsilon} &  \\ & & E_{\epsilon} &  A_{\epsilon} \\ C_{\sigma}J_{\sigma}^3 & & & C_{\epsilon}   \\& & & E_{\epsilon}  \\J_{\sigma}^4 \end{bmatrix} \\
 &=& \rank E_{\epsilon} + 4 \rank I_f + 4 \rank E_{\eta} + 3 \rank I_{\sigma} + \rank \begin{bmatrix} C_{\sigma}J_{\sigma} & C_1 & C_2 &  \\ & I_{\epsilon} & 0 & A_1 & A_2 \\ C_{\sigma}J_{\sigma}^2 & & & C_1 & C_2   \\ & & & I_{\epsilon} & 0 & A_1 & A_2 \\ C_{\sigma}J_{\sigma}^3 & & & & & C_1 & C_2   \\&&& & & I_{\epsilon} & 0 \\J_{\sigma}^4 \end{bmatrix} .
\end{eqnarray*}
Therefore, by multiplying the last matrix in the equality above with $\hat{U}$ from the left and applying Lemma~\ref{lm:rank} we may infer that
\begin{eqnarray}\label{r3}
	\rank \mathcal{F}_4 = 4 \rank E_{\epsilon} + 4 \rank I_f + 4 \rank E_{\eta} + 3 \rank I_{\sigma} + \rank \begin{bmatrix} C_{\sigma}J_{\sigma} & C_2 & C_1A_2 & C_1A_1A_2 \\ C_{\sigma}J_{\sigma}^2 & & C_2 & C_1A_2 \\ C_{\sigma}J_{\sigma}^3 & & & C_2 \\ J_{\sigma}^4 & &  \end{bmatrix} .
\end{eqnarray}
By a similar calculation as above we find that
\begin{eqnarray}\label{r4}
\rank \mathcal{F}_{4,L} = 4 \rank E_{\epsilon} + 4 \rank I_f + 4 \rank E_{\eta} + 3 \rank I_{\sigma} + \rank \begin{bmatrix} C_{\sigma}J_{\sigma} & C_2 & C_1A_2 & C_1A_1A_2 \\ C_{\sigma}J_{\sigma}^2 & & C_2 & C_1A_2 \\ C_{\sigma}J_{\sigma}^3 & & & C_2 \\ L_{\sigma}J_{\sigma} & L_2 & L_1A_2 & L_1A_1A_2 \\ C_{\sigma}J_{\sigma}^2 & & L_2 & L_1A_2 \\ L_{\sigma}J_{\sigma}^3 & & & L_2 \\ J_{\sigma}^4 & & \end{bmatrix} .
\end{eqnarray}

\noindent Therefore, in view of Lemma \ref{lm:ker}, \eqref{r3} and \eqref{r4} provide that
 $\rank \mathcal{F}_4 = \rank \mathcal{F}_{4,L}$ if, and only if,
\begin{equation}\label{R5}
\ker \begin{bmatrix} C_{\sigma}J_{\sigma} & C_2 & C_1A_2 & C_1A_1A_2 \\ C_{\sigma}J_{\sigma}^2 & & C_2 & C_1A_2 \\ C_{\sigma}J_{\sigma}^3 & & & C_2 \\ J_{\sigma}^4 & &  \end{bmatrix} \subseteq \ker \begin{bmatrix} L_{\sigma}J_{\sigma} & L_2 & L_1A_2 & L_1A_1A_2 \\ C_{\sigma}J_{\sigma}^2 & & L_2 & L_1A_2 \\ L_{\sigma}J_{\sigma}^3 & & & L_2 \end{bmatrix}.
\end{equation}

\noindent Then by repeating the procedure that produces \eqref{R1}, \eqref{R2}, and \eqref{R5}, it is easy to prove that for any integer $l \geq 3$, $\rank \mathcal{F}_{l} = \rank \mathcal{F}_{l,L}$ if, and only if,
\begin{equation}\label{R3}
\ker \begin{bmatrix}
	C_{\sigma}J_{\sigma} & C_2 & C_1A_2 & C_1A_1A_2 & \hdots &  C_1A_1^{l-3}A_2  \\ C_{\sigma}J_{\sigma}^2 & & C_2 & C_1A_2 &  \hdots & C_1A_1^{l-4}A_2  \\
	\vdots & &  & \ddots & \ddots &  \vdots  \\
	C_{\sigma}J_{\sigma}^{l-2} & & & & C_2 & C_1A_2 \\
	C_{\sigma}J_{\sigma}^{l-1} & & & & & C_2 \\
	J_{\sigma}^{l} & & & & &  \end{bmatrix} \subseteq \ker \begin{bmatrix}
	L_{\sigma}J_{\sigma} & L_2 & L_1A_2 & L_1A_1A_2 & \hdots &  L_1A_1^{l-3}A_2  \\ L_{\sigma}J_{\sigma}^2 & & L_2 & L_1A_2 &  \hdots & L_1A_1^{l-4}A_2  \\
	\vdots & &  & \ddots & \ddots &  \vdots  \\
	L_{\sigma}J_{\sigma}^{l-2} & & & & L_2 & L_1A_2 \\
	L_{\sigma}J_{\sigma}^{l-1} & & & & & L_2
   \end{bmatrix}.
\end{equation}

With these findings we are now ready to state the main result of this paper.
	
\begin{theorem}\label{thm2}
For a given system \eqref{dls1}, the following statements are equivalent:
\begin{enumerate}[(a)]
	\item $(E,A,C)$ is partially impulse observable with respect to $L$. \label{a}
	\item The condition \eqref{eq:kimpulse} holds. \label{b}
	\item $\rank \mathcal{F}_{n+1} = \rank \mathcal{F}_{n+1,L}$, \label{c}
	\item $\mathcal{W}^*_{[\bar{E},\bar{A}]} \cap \bar{A}^{-1}(\im{\bar{E}}) \subseteq \ker L$. \label{d}		
\end{enumerate}
\end{theorem}

\begin{proof}
The equivalence of (\ref{a}) and (\ref{b}) is a direct consequence of Theorem \ref{thm1} and the conditions~\eqref{R1} and~\eqref{R3}. The statement (\ref{b}) $\Rightarrow$ (\ref{c}) is obvious. Thus, in order to complete the proof, it is sufficient to show that (\ref{c}) $\Rightarrow$ (\ref{d}) and (\ref{d}) $\Rightarrow$ (\ref{b}). Before proving these statements, we observe, by a simple permutation of rows, that
\begin{equation*}
\mathcal{F}_{l,L} = P \begin{bmatrix} \mathcal{F}_l \\ \NiceMatrixOptions
{nullify-dots,code-for-last-col = \color{black},code-for-last-col=\color{black}}
\begin{bNiceMatrix}[first-row,last-col]
& \Ldots[line-style={solid,<->},shorten=0pt]^{l \text{ block columns}} \\ 0 & L &&\\ &&\ddots& \\ &&&L & ~\Vdots[line-style={solid,<->},shorten=0pt]^{(l-1) \text{ block rows}}~
\end{bNiceMatrix}\\ \phantom{o}
\end{bmatrix},
\end{equation*}
where $P$ is a suitable permutation matrix,
and hence $\rank \mathcal{F}_{l} = \rank \mathcal{F}_{l,L}$ holds if, and only if,
\begin{equation}\label{eq:ker-Fl-ker-L}
\ker \mathcal{F}_l \subseteq \ker  \begin{bmatrix} 0 & L &&\\ &&\ddots& \\ &&&L
\end{bmatrix} = \mathbb{R}^n \times \underset{(l-1) \text{ times}}{\underbrace{\ker L \times \ldots \times \ker L}}.
\end{equation}

(\ref{c}) $\Rightarrow$ (\ref{d}): Let $v_n \in \mathcal{W}^*_{[\bar{E},\bar{A}]} \cap \bar{A}^{-1}(\im{\bar{E}})$. Since the Wong sequences terminate after finitely many steps, and in each iteration before termination the dimension increases by at least one, it is clear that $\mathcal{W}^*_{[\bar{E},\bar{A}]} = \mathcal{W}^n_{[\bar{E},\bar{A}]}$. Hence there exists $v_{n-1} \in \mathcal{W}^{n-1}_{[\bar{E},\bar{A}]}$ such that $\bar E v_n = -\bar A v_{n-1}$. Successively, there exist $v_i \in \mathcal{W}^i_{[\bar{E},\bar{A}]}$ such that $\bar E v_{i+1} = -\bar A v_i$ for $i=n-2,\ldots,1$ and $\bar E v_1 = 0$, since $v_1 \in \mathcal{W}^1_{[\bar{E},\bar{A}]} = \ker \bar E$. Furthermore, since also $v_n \in \bar{A}^{-1}(\im{\bar{E}})$ there exists $v_{n+1}\in\mathbb{R}^n$ such that $\bar A v_n = - \bar E v_{n+1}$. Therefore, we find that $(v_{n+1}^\top, v_n^\top, \ldots, v_1^\top)^\top \in \ker \mathcal{F}_{n+1}$ and from~\eqref{eq:ker-Fl-ker-L} it follows that $v_n \in \ker L$.

(\ref{d}) $\Rightarrow$ (\ref{b}): In order to show (\ref{b}) we prove that \eqref{eq:ker-Fl-ker-L} holds for all $l\ge n+1$. Let $x = (x_l^\top,\ldots,x_1^\top)^\top\in\ker \mathcal{F}_l$. Then $\bar E x_l = -\bar A x_{l-1}, ~\hdots,~ \bar E x_2 = -\bar A x_1,~ \bar E x_1 = 0$ and hence we have
\begin{align*}
x_1 &\in \ker \bar E = \mathcal{W}^1_{[\bar{E},\bar{A}]},\\
x_2 &= \bar{E}^{-1}(-\bar A x_1) \in \bar{E}^{-1}(\bar A \mathcal{W}^1_{[\bar{E},\bar{A}]}) = \mathcal{W}^2_{[\bar{E},\bar{A}]},\\
&\ \, \vdots  \\
x_l &= \bar{E}^{-1}(-\bar A x_{l-1}) \in \bar{E}^{-1}(\bar A \mathcal{W}^{l-1}_{[\bar{E},\bar{A}]}) = \mathcal{W}^l_{[\bar{E},\bar{A}]}.
\end{align*}
Since $\mathcal{W}^i_{[\bar{E},\bar{A}]} \subseteq \mathcal{W}^*_{[\bar{E},\bar{A}]}$ for all $i\ge 1$ we have that $x_i \in \mathcal{W}^*_{[\bar{E},\bar{A}]}$ for all $i=1,\ldots,l$. Furthermore, since $\bar A x_{i} = - \bar E x_{i+1}$ for $i=1,\ldots,l-1$ we have that $x_i \in \bar{A}^{-1}(\im{\bar{E}})$, hence
\begin{equation*}
\forall\, i=1,\ldots,l-1:\ x_i \in \mathcal{W}^*_{[\bar{E},\bar{A}]} \cap \bar{A}^{-1}(\im{\bar{E}}) \subseteq \ker L,
\end{equation*}
which shows \eqref{eq:ker-Fl-ker-L}. This completes the proof. $\hfill \Box$
\end{proof}

\begin{remarks}\label{Iobsvw}
It is clear that if $L = I_n$, then the condition of statement (\ref{d}) in Theorem \ref{thm2} reduces to the criterion for I-observability from Lemma~\ref{lm4}.
\end{remarks}

\begin{remarks}\label{rem7}
The condition of statement (\ref{d}) in Theorem \ref{thm2} is straightforward to implement by using a one-line command, for instance, in MATLAB. If $s$ is the least positive integer such that $\mathcal{W}^{s+1}_{[\bar{E},\bar{A}]} = \mathcal{W}^s_{[\bar{E},\bar{A}]}$, then the number $(n+1)$ in statement (\ref{c}) of Theorem \ref{thm2} can be replaced by $s$. Here, we use $(n+1)$ blocks in $\mathcal{F}$ because the value of $s$ is not known in advance and our main aim is to provide a condition directly in terms of the known data, i.e., the system coefficient matrices and the dimension $n$. Notably, using $(n+1)$ blocks does not make the condition of statement (\ref{c}) in Theorem~\ref{thm2} less or more restrictive.
\end{remarks}

\section{Illustrative examples}\label{numerical}

\begin{example}\label{exp1}
Consider system \eqref{dls1} described by the coefficient matrices
$$E = \begin{bmatrix} 1 & 0 \end{bmatrix},~
A = \begin{bmatrix} 0 & 1 \end{bmatrix},~
C = \begin{bmatrix} 1 & 0 \end{bmatrix},~ \text{and }
L = \begin{bmatrix}  0 & 1 \end{bmatrix}.$$
Then $(E,A,C)$ is not partially impulse observable with respect to $L$, because choosing $x_1$ as the Heaviside step function and $x_2 = \delta$ we obtain a solution with $y[t]=0$ for all $t\ge 0$, but $z[0] = x_2[0] = \delta \neq 0$, thus $z$ exhibits impulses while $y$ is impulse free. On the other hand it is easy to verify that
	\begin{equation*}
	\rank \mathcal{F}_{3}  = 4 \neq 5 =  \rank \mathcal{F}_{3,L}.
	\end{equation*}
\end{example}

\begin{example}\label{exp2}
Consider system \eqref{dls1} with the same coefficient matrices as in Section \ref{count}. Then, as shown there, $(E,A,C)$ is not partially impulse observable with respect to $L$. It is easy to see that
	\begin{equation*}
	\rank \mathcal{F}_{4}  = 9 \neq 10 =  \rank \mathcal{F}_{4,L}.
	\end{equation*}
\end{example}

\begin{example}\label{exp3}
Consider system \eqref{dls1} described by the coefficient matrices
\begin{equation*}
E = \begin{bmatrix} 0 & 1 & 0  \\ 0 & 0 & 1  \\ 0 & 0 & 0  \end{bmatrix},~
A = \begin{bmatrix} 1 & 0 & 0  \\ 0 & 1 & 0  \\ 0 & 0 & 1  \end{bmatrix},~
C = \begin{bmatrix} 1 & 0 & 0 \end{bmatrix}, \text{ and } L = \begin{bmatrix} 0 & 1 & 0 \end{bmatrix}.
\end{equation*}
As compared to Example \ref{exp2}, here we show that a suitable change in the output matrix $C$ makes the system partially impulse observable with respect to the same matrix $L$. By using \eqref{soln}, it is clear that
\[
		x|_{[0,\infty)} = -\begin{bmatrix} x^0_2 \\ x^0_3 \\ 0 \end{bmatrix}  \delta - \begin{bmatrix} x^0_3 \\ 0 \\ 0 \end{bmatrix}\dot \delta,
\]
for suitable $x^0_2, x^0_3$, and hence
    \begin{align*}
       y[0] &= -x_2^0 \delta - x_3^0\dot \delta, \\
		z[0] &= - x_3^0\delta.
	\end{align*}
Clearly $y[0] =0$ implies $x^0_2 = x^0_3 = 0$ and hence also $z[0] = 0$. Thus $(E,A,C)$ is partially impulse observable with respect to $L$. We can also verify this fact by checking the rank condition
\begin{equation*}
	\rank \mathcal{F}_{4}  = 11 =  \rank \mathcal{F}_{4,L}.
\end{equation*}
This demonstrates the effectiveness of statement (\ref{c}) in Theorem \ref{thm2}.
\end{example}

\section{Conclusion}\label{conc}
This paper has established necessary and sufficient conditions for the partial impulse observability of linear descriptor systems. The developed conditions in terms of a rank criterion involving the original system coefficient matrices and the Wong sequences, respectively, are very easy to implement.

\section*{Declaration of competing interest}
The authors declare that they have no known competing financial interests or personal relationships that could have appeared to influence the work reported in this paper.

\section*{Acknowledgment}
The First author is thankful to Council of Scientific and Industrial Research, New Delhi, for the award of JRF through grant number $09/1023(0022)/2018$-EMR-I. The Science and Engineering Research Board, New Delhi, supported the third author under project MTR/$2019/000494$.

\end{document}